\documentclass[11pt]{amsart}

\usepackage{amssymb}
\usepackage{graphics}
\usepackage{graphicx}
\usepackage{amscd}
\usepackage{amsmath, epsfig,  psfrag}

\newtheorem{theorem}{Theorem}
\newtheorem{lemma}[theorem]{Lemma}
\newtheorem{prop}[theorem]{Proposition}

\newtheorem{cor}[theorem]{Corollary}

\theoremstyle{definition}
\newtheorem{remark}[theorem]{Remark}
\newtheorem{example}[theorem]{Example}

\numberwithin{theorem}{section}

\newcommand{\HF}{\widehat{HF}}
\newcommand{\on}{\operatorname}

\renewcommand{\d}{\partial}
\newcommand{\Map}{\on{Map}}
\newcommand{\PBraid}{\text{PBraid}}
\newcommand{\Spinc}{\on{Spin}^c}

\newcommand{\ZZ}{\mathbb{Z}}
\newcommand{\RR}{\mathbb{R}}

\newcommand\goth[1]{\mathfrak{#1}}
\newcommand{\s}{\goth{s}}
\renewcommand{\t}{\goth{t}}

\newcommand{\sign}{\on{sign}}

\newcommand{\unknot}{\mbox{unknot}}

\newcommand{\tb}{\on{tb}}
\newcommand{\rot}{\on{rot}}

\begin{document}

\author{Olga Plamenevskaya}
\address{Department of Mathematics, Stony Brook University, Stony Brook, NY 11794}
\email{olga@math.sunysb.edu}
\thanks{Partially supported by 
NSF grant DMS-0805836.}
\title{On Legendrian surgeries between lens spaces}

\begin{abstract} We obtain some obstructions to existence of Legendrian surgeries between tight lens spaces. We also study  
Legendrian surgeries between overtwisted contact manifolds.
\end{abstract}

\maketitle

\section{Introduction}

The focus of this paper is the following basic question:
given two closed contact 3-manifolds $(Y_1, \xi_1)$ and $(Y_2, \xi_2)$, can a Legendrian surgery on a link in $(Y_1, \xi_1)$ 
produce $(Y_2, \xi_2)$? If so, what can be said about the number of components of the surgery link? What happens if we don't fix contact 
structures $\xi_1$ and $\xi_2$, but just want to obtain, say, a tight $Y_2$ from a tight $Y_1$?
 
It is well known that if 
$(Y_2, \xi_2)$ is obtained from $(Y_1, \xi_1)$ by a surgery on Legendrian link, the smooth cobordism between  $(Y_1, \xi_1)$  and $(Y_2, \xi_2)$
has a Stein structure \cite{E1}.  Such cobordisms were studied by Etnyre and Honda \cite{EH}; in particular, they showed a Stein cobordism 
between  $(Y_1, \xi_1)$ and $(Y_2, \xi_2)$ exists whenever $(Y_1, \xi_1)$ is overtwisted. 
If $(Y_1, \xi_1)$ is Stein fillable, a Legendrian surgery produces a Stein filling of $(Y_2, \xi_2)$, and one can try to find obstructions to 
Legendrian surgeries by considering Stein fillings. From known classification theorems \cite{El, McD, Lis} we  immediately see, for example, 
that standard contact structures on $L(p,1)$ cannot be obtained from any lens space except $S^3$ (see Section \ref{fillings} for more results 
of this sort). 
 More interestingly, we
examine Stein fillings of planar open books to prove 
\begin{theorem} \label{self}
No tight lens space can be obtained from itself by Legendrian surgery on a link.
\end{theorem} 

By contrast, it is easy to show that every overtwisted contact manifold can be obtain from itself by a Legendrian surgery on an appropriate 
Legendrian unknot (see Proposition \ref{cando}). 

Theorem \ref{self}   applies more generally to fillable contact 3-manifolds admitting planar open books. To prove the theorem, we show that for any 
contact structure supported by a planar open book, there is a bound on the second Betti number of its fillings (see Corollary \ref{bound}).

When a Stein cobordism from $(Y_1, \xi_1)$ to $(Y_2, \xi_2)$ exists, one can wonder about its Betti numbers  
(note that cobordisms we consider arise from surgeries, and consist of 2-handles only). In particular, we can ask whether $(Y_2, \xi_2)$ can be obtained from
$(Y_1, \xi_1)$ by a Legendrian surgery on a single knot; for the question to be meaningful,  
we will assume that obvious obstructions to existence of such surgery, both from the topological viewpoint (eg from the singular 
homology of $Y_1$ and $Y_2$) and from the homotopy classes of contact structures all vanish.
Of course, this question is closely related to the question of integral knot surgery:
can $Y_2$ be obtained from $Y_1$ by an integral surgery on a single knot? Such knot surgery questions have been studied by means of many tools. An important classical result, the knot 
complement theorem of Gordon--Luecke \cite{GL}, asserts that if $S^3$ is obtained by a non-trivial surgery on a knot $K \subset S^3$, then $K$ 
is an unknot.  More recently, tools from the Heegaard Floer theory led to substantial progress \cite{Hed, OSlens, Ras} on understanding     
classical knots with lens space surgeries, and enumerating lens spaces that can be obtained from $S^3$ by a surgery on a single knot (Berge's conjecture).  These results, together with considerations of the Thurston--Bennequin number, give some immediate easy results for Legendrian 
knot surgeries (see Section \ref{tb}).  Further, one expects that there are additional obstructions for surgeries between tight contact structures.
Tight lens spaces are the easiest example because they have a complete classification, are all Stein fillable, and have simplest possible 
Heegaard Floer homology. Using Heegaard Floer contact invariants, we get    

\begin{theorem} \label{HFlo} No tight contact structure on $L(p_2, q_2)$ can be obtained from a tight contact $L(p_1, q_1)$ by a Legendrian surgery 
on a single knot if $p_1, p_2$ are coprime,  $-p_1 q_2$  is a square modulo $p_2$ and $-1$ is not.
\end{theorem}  

(We follow the notation convention where $L(p,q)$ stands for the oriented manifold which is a $-p/q$ surgery on the unknot in $S^3$. Note that this 
is opposite to conventions in \cite{OSlens, Ras}. We always assume that $p>q>0$.)

Note that Theorem \ref{HFlo} does not contain any information about the existence of topological surgery. In some cases, 
an integral surgery between two lens spaces exists but can't be made into a Legendrian surgery between tight contact structures; 
 in other cases, we can only claim that a Legendrian surgery between 
tight lens spaces can't exist, and it is possible that there's no integral surgery whatsoever. (We give examples in Section \ref{hfloer}.)

It is also useful to observe that  for small lens spaces, one can rule out Legendrian knot surgeries between tight contact structures
simply by computing all of their 3-dimensional homotopy invariants. However, there are  examples where Theorem \ref{HFlo}
gives an obstruction to the existence of Legendrian surgery, but the homotopy invariants don't.



Next, we turn attention to overtwisted manifolds. By \cite{EH}, any contact manifold $(Y_2, \xi_2)$ can be obtained from an overtwisted  
$(Y_1, \xi_1)$  by a Legendrian 
surgery on a link. A priori, this Legendrian link may have more components than a smallest link in $Y_1$ with an integral surgery producing $Y_2$. 

We have 
\begin{theorem} Suppose $Y_2$ can be obtained from $Y_1$ by an integral surgery on $m$-component link. Let $\zeta$ be a tight, $\eta$ an overtwisted 
contact structure on $Y_2$.  Then $(Y_2, \eta)$ can be obtained from some overtwisted contact structure on $Y_1$ by a Legendrian surgery on an $m$-component link;  $(Y_2, \zeta)$ can be obtained from some overtwisted $Y_1$ by a Legendrian surgery on an $m+1$ component link.
\end{theorem}

We  establish a variety of other related results (many of them mere observations) in Section \ref{otwist}. 
We also give examples where $Y_1$, $Y_2$ are related by an integral surgery on a single knot, but a two-component Legendrian link  
is needed to produce any tight $Y_2$ from any overtwisted $Y_1$ by a Legendrian surgery.

We will assume that the reader is familiar with  basic contact topology constructions and with Heegaard Floer contact invariants,  
and state the necessary facts only very briefly.  
   
\subsection*{Acknowledgements} I am grateful to John Etnyre and Jeremy Van Horn-Morris for some helpful conversations, and Chris Wendl for his useful
comments. I would also like to thank the referee for pointing out a sign mistake in Section \ref{hfloer}.
   

\section{Obstructions from Stein fillings} \label{fillings}

The first  way to find some obstructions for Stein cobordisms between lens spaces is provided by examining Stein fillings. Let $(L(p,q), \xi_{std})$ 
denote the (tight) contact lens space obtained as a quotient of $(S^3, \xi_{std})$. Symplectic fillings of $(L(p, q), \xi_{std})$ were classified by 
Lisca (up to orientation-preserving diffeomorphism and blow-up) \cite{Lis}; earlier results were obtained by Eliashberg for fillings of $(S^3, \xi_{std})$ \cite{El}, and by McDuff for $(L(p, 1), \xi_{std})$ \cite{McD}. We collect the statements about lens spaces with few fillings below; note that Lisca's results yield classification 
of Stein fillings up to diffeomorphism, since Stein surfaces are minimal. 
\begin{itemize}
\item $(S^3, \xi_{std})$ has a unique Stein filling, namely $D^4$.
\item For $p\neq 4$, $(L(p, 1), \xi_{std})$ has a unique filling, namely $D_{-p}$, the disk bundle over $S^2$ with Euler number $-p$.  
$(L(4,1), \xi_{std})$ has two fillings, $D_{-4}$ and another which is a rational homology ball.
\item $(L(p, p-1), \xi_{std})$ has a unique filling given by the diagram on Fig \ref{p,p-1}; note also that $L(p, p-1)$ carries a unique tight contact structure \cite{Ho}.
\item All lens spaces $(L(p, q), \xi_{std})$ have finitely many Stein fillings.
\end{itemize}

\begin{figure}[htb] 
\includegraphics[scale=0.8]{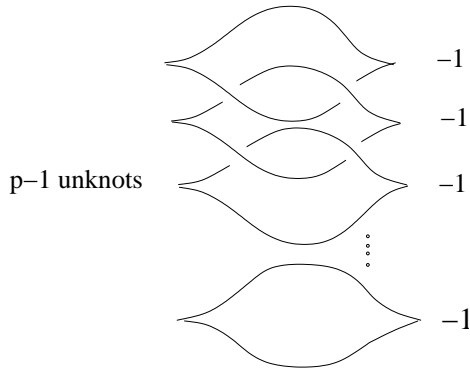}
\caption{The unique Stein filling of the tight $L(p, p-1)$. (The numbers indicate surgery coefficients relative to the contact framing.)}\label{p,p-1} 
\end{figure}

Since a Stein filling of $L(p_1, q_1)$ together with a Stein cobordism from $L(p_1, q_1)$ to  $L(p_2, q_2)$ forms a Stein 
filling of $L(p_2, q_2)$, the above results immediately give some obstructions to existence of Stein cobordisms. We can state them for both 
the standard contact structures and their conjugates, since a Stein filling $(X, J)$ of $(Y, \xi)$ induces a Stein filling 
$(X, \bar{J})$ of $(Y, \bar{\xi)}$. We also recall that  there are exactly two universally tight contact structures on $L(p, 1)$: $\xi_{std}$ and its conjugate. The proposition below holds because otherwise there would be non-standard fillings (with the second Betti number not allowed by the classification results above); to construct these non-standard fillings in (1) and (2),  start with the fillings for $L(p,q)$ that have no 1-handles. 

\begin{prop} (1)  a tight $S^3$ cannot be obtained from any other tight lens space by a Legendrian surgery on a link.
 
 (2) Universally tight $L(p, 1)$  can be obtained only from $(S^3, \xi_{std})$  (by a Legendrian surgery on a knot); it cannot be 
 obtained from any other tight lens space by a Legendrian surgery on any link.
 
 (3) A tight $L(n, n-1)$ can be obtained from a tight $L(m, m-1)$ by a Legendrian surgery on a knot if and only if $m=n-1$; 
 more generally,   a tight $L(n, n-1)$ can be obtained from a tight $L(m, m-1)$ by a Legendrian surgery on a link if and only if $m<n$.
 
\end{prop}
\qed

In fact, part (2) can be extended: in a joint work with Jeremy Van Horn-Morris \cite{PV} we show that every virtually overtwisted tight contact 
structure on $L(p,1)$ has a unique filling; it then follows that {\em any} tight  $L(p, 1)$  can be obtained only from $(S^3, \xi_{std})$. 

We now establish the main result of this section. 

\begin{theorem}\label{planar} Let $(Y, \xi)$ be a Stein fillable contact manifold that admits a compatible planar open book decomposition.
Then  $(Y, \xi)$ cannot be obtained from itself by Legendrian surgery on any link.
\end{theorem} 

We recall that all tight contact structures on lens spaces admit planar open books \cite{Schon}.
Thus, Theorem \ref{self} is an immediate corollary of Theorem \ref{planar}.

\begin{proof} We will use  a theorem of Wendl \cite[Theorem 1]{We} as our key tool. Wendl essentially proves that every Stein filling 
of a contact structure with a planar open book decomposition admits (after an ``enlargement'' of the symplectic structure) a compatible positive allowable Lefschetz fibration that induces 
the given open book on the boundary. Topologically, this means that every Stein filling is diffeomorphic to 
a Lefschetz fibration that corresponds to the product of 
positive Dehn twists in some factorization of the monodromy for a {\em given} open book.  
We show  that this leads to an upper bound on Betti numbers of fillings. Indeed, suppose the page of the open book is  a disk with $m$ holes,  
$P_m = {D}- (D_1\cup D_2 \dots  \cup D_m)$; the monodromy $f$ is an 
an element of the mapping class group  $\Map(P_m)$ of the disk with $m$ holes.  

Shrinking the holes and thinking of them as punctured points, we get a homomorphism $\Phi: \Map(P_m) \to \PBraid_m$
into the pure braid group on $m$ strings. (Denh twists around the holes become trivial in $\PBraid_m$.)   Let $\Sigma: \PBraid_m \to \ZZ$ be given by the algebraic crossing number of braids; 
then $\Sigma \circ \Phi$ will be an invariant of the monodromy of our open book. Now, suppose that the monodromy $f$ is expressed 
as a product of positive Dehn twists around some simple closed curves in ${D}- (D_1\cup D_2 \dots  \cup D_m)$. We can consider the number of punctures 
enclosed by each curve, since these are simple closed curves in the disk, so the complement of each has the ``inside'' and the ``outside''
component. Observe that each Dehn twist that encloses more than one puncture contributes positively to the algebraic crossing 
number $(\Sigma \circ \Phi) (f)$ of the corresponding pure braid; this gives an upper bound on the number of such Dehn twists. 
To find an upper bound on the number of Dehn twists around the holes, pick a hole $D_i$ and consider 
a self-homeomorphism of the $m$-holed disk that exchanges the roles of $\partial D_i$ and $\partial  {D}$. 
The previous argument  now yields a bound on the number of Dehn twists around the outer component of the boundary, 
i.e. around the hole $D_i$. (The image of the monodromy in $\PBraid_m$ depends on the choice of the outer component, 
so we will get different bounds for different $i$.) We now have a bound on the total number of positive Dehn twists in decompositions 
of the monodromy, i.e. on the second Betti number of fillings of the contact structure compatible with $(P_m, f)$. 
   
  To complete the proof, observe that  if a contact manifold  can be obtained from itself by some Legendrian surgery, we could generate fillings with arbitary large $b_2$ by starting with an arbitary filling and adding multiple copies of  the collection of 2-handles that has trivial effect on the contact boundary.   
\end{proof}

The proof of Theorem \ref{planar} yields the following 

\begin{cor} \label{bound} If $(Y, \xi)$ admits a planar open book, there is a number $M= M(Y, \xi)$ such that 
for any Stein filling $X$ of $(Y, \xi)$, the bound $b_2(X) < M$ holds.  
\end{cor}

\begin{remark} It is known \cite{Rong} that for lens spaces, the only true cosmetic Dehn surgeries 
(i.e. rational knot surgeries that preserve the oriented manifold)   
are those that ``switch'' the roles of the solid tori in a genus 1 Heegaard decomposition of the lens space. (More precisely, 
one of these solid tori in $L(p,q)$ is reglued to produce $L(p,q')$ with $qq' \equiv 1$ modulo $p_2$.) 
One can easily check that these surgeries can never be integral; thus the claim of Theorem~\ref{self} is moot for
Legendrian surgeries on {\em knots}.        
\end{remark} 

We also state an obvious corollary of Theorem \ref{self}:

\begin{cor} \label{back} If  a Legendrian surgery on a link in a tight $(L(p_1, q_1), \xi_1)$ produces a tight  $(L(p_2, q_2), \xi_2)$, 
then  no Legendrian surgery in $(L(p_2, q_2), \xi_2)$ can produce $(L(p_1, q_1), \xi_1)$.
\end{cor} 

\section{Obstructions from Heegaard Floer contact invariants} \label{hfloer} 

In this section, we  use the Heegaard Floer contact  invariants \cite{contOS} to find obstructions to existence of a Legendrian knot surgery between two tight lens spaces. We work with $\ZZ/2\ZZ$ coefficients throughout.
Recall that the invariant $c(\xi)$ is an element of $\HF(-Y)$ for a contact 3-manifold $(Y, \xi)$; if $\xi$ is Stein-fillable, $c(\xi)\neq 0$.
(Thus the contact invariant is non-zero for all tight lens spaces.) 

If $(Y_2, \xi_2)$ is obtained from $(Y_1, \xi_1)$ by  Legendrian surgery,
the surgery cobordism carries a canonical $\Spinc$ structure $\t$ induced by the Stein structure on the 2-handle. 
We can consider the map   $F_{\t}:\HF(-Y_2, \t|_{-Y_2} ) \to \HF(-Y_1, \t|_{-Y_2})$ induced by the $\Spinc$ cobordism, 
as well as the map  $F:\HF(-Y_2) \to \HF(-Y_1)$ obtained by summing over all $\Spinc$ structures on $W$. Then by  \cite{contOS,Gh,LS5},
$$F(c(\xi_2))= F_{\t}(c(\xi_2))=c(\xi_1).$$  

Thus, we can show that there is no 
Legendrian knot surgery from Stein-fillable $(Y_1, \xi_1)$ to $(Y_2, \xi_2)$ if we prove that any single 2-handle attachment producing  
$-Y_1$ from $-Y_2$ induces a vanishing map $F$ on Heegaard Floer homology. We will follow this strategy to prove  Theorem \ref{HFlo}. 

To prove that certain maps associated to surgeries vanish, we will make use of the surgery exact triangle.
Recall that the exact triangle relates the Heegaard Floer homologies of a 3-manifold $Y$ and the following surgeries.
For a choice of longitude $l$ for a knot $K \subset Y$, 
let $Y_l$ be the manifold obtained by the surgery that attaches the meridian of the surgery torus  to $l$; for $k \in \ZZ$, let 
$Y_{l+k\,m}$ the manifold obtained by sending the meridian  to $l+k\, m$. Then the homologies of the three manifolds $Y$, $Y_l$, $Y_{m+l}$ 
fit into an exact triangle: 
$$
\dots \to \HF(Y) \to \HF (Y_l) \to \HF (Y_{m+l}) \to \HF(Y) \to \dots
$$

We are interested in lens spaces, which have the simplest possible Heegaard Floer homology; indeed, if $Y$ is a lens space, 
$
\HF(Y,\s) = \ZZ/2\ZZ 
$
in each $\Spinc$ structure $\s$ on $Y$. ( A rational homology sphere with this property is called an $L$-space.)
Recall the following
\begin{lemma}\label{Lsp} \cite{OSlens} If $Y$, $Y_l$,  and $Y_{m+l}$ are members of a surgery triple such that
$Y$ and $Y_l$ are $L$-spaces, and 
\begin{equation}
\label{h+h} 
|H_1(Y_{l+m})|= |H_1(Y_{l})| + |H_1(Y)|,
\end{equation}
then $Y_{m+l}$ is also an $L$-space. Moreover, the map
$
F: \HF(Y) \to \HF (Y_l),
$      
induced by the surgery cobordism, is identically zero.
\end{lemma}

\begin{proof}[Proof of Theorem \ref{HFlo}] Suppose that the space $Y_2= L(p_2, q_2)$ is obtained by integral surgery 
on a knot in $Y_1=L(p_1, q_1)$. The surgery cobordism from $Y_1$ to $Y_2$, consisting of one 2-handle, can be considered as cobordism 
from $-Y_2$ to $-Y_1$; equivalently, we can say that $-Y_1= -L(p_1, q_1)$ is obtained by integral surgery on 
a knot $K \subset -Y_2=-L(p_2, q_2)$. Let $l$ be the choice of longitude for $K$ corresponding to the surgery framing.
Since $p_1$ and $p_2$ are assumed coprime, $[K]$ generates $H_1(-Y_2)$. 

Let $C$ be the complement of a tubular neighborhood of $K$ in $-Y_2$.
The homology  $H_1(\d C)$ of its torus boundary is generated by the chosen longitude $l$ and the meridian $m$ of the knot $K$. 
There exists an essential simple closed curve in $\d C$ that bounds in $C$. After an appropriate choice of orientations, such curve belongs to a (uniquely determined) homology class of the form  $a\,m + p_2\, l \in H_1(\d C)$. The quantity $a \mod p_2$ 
is the self-linking number $K\cdot K$, which is an invariant of $K$. We will, however, consider $a$ as an integer; 
note that $a$ is well-defined once the longitude $l$ is fixed. Observe also that $a$ and $p_2$ are coprime.      

We perform surgery on $K\subset -Y_2$ and consider the surgery triple $Y=-Y_2$, $Y_l=-Y_1$, and $Y_{l+m}$.  
The integer $a$ determines the size of homology of integral surgeries on $K$. Indeed,
we have $|H_1(Y_{l+k\,m})|= |-a +k p_2|$, where we assume that $a$ is computed with respect to the chosen 
longitude $l$. (This is a standard calculation; see e.g. \cite{Ras} for a quick review with proofs.) 
Thus,  $|H_1(Y_l)|=|-a|$ and $|H_1(Y_{l+m})|= |-a +p_2|$.

In our setup, $Y_l=-L(p_1, q_1)$, so $a=\pm p_1$. We next determine the sign of $a$. Indeed, recall that 
for a knot $K \subset -L(p_2, q_2)$ the self-linking number $K\cdot K$ can take values $k^2 q' \mod p_2$, 
where $q'$ is the multiplicative inverse of $q_2$ modulo $p_2$, and $k$ is an integer. Then, $a q_2$ is a square modulo 
$p_2$. The hypotheses of Theorem \ref{HFlo} imply that $-p_1 q_2$ is a square $\mod p_2$ but $p_1 q_2$ is not; thus, 
$a=-p_1$.    

Since $p_1, p_2>0$, we now have $|H_1(Y_{l+m})|=  p_1+p_2$, $|H_1(Y)|=p_2$, $|H_1(Y_l)|=p_1$. So
the condition (\ref{h+h}) is satified, and Lemma \ref{Lsp} implies that the map $F:\HF(-L(p_2, q_2)) \to \HF(-L(p_1, q_1))$, 
induced by the cobordism, is trivial. As explained in the beginning of the section, it follows that a tight $L(p_2, q_2)$ 
cannot be obtained from a tight $L(p_1, q_1)$ by a Legendrian surgery. \end{proof}
 
\begin{remark} When $p_2$ is prime, the hypotheses of Theorem \ref{HFlo} are equivalent to requiring that
$p_2$ be congruent to $3 \mod 4$, and exactly one of the numbers $p_1$, $q_2$ 
be a square $\mod p_2$.
\end{remark}

\begin{example} For the lens spaces $Y_1=L(3r-1, q)$ and $Y_2=L(3,1)$ the hypotheses of the theorem hold for any $r,q>0$, thus none of the two tight contact structures $\zeta_1, \zeta_2$ on $L(3,1)$ can be obtained from any tight contact structure on $L(3r-1, q)$ by a Legendrian surgery. Note, however, 
that there is an obvious integral surgery between  $Y_1=L(3r-1, r)$ and $L(3,1)$.
In fact, by classification of tight contact structures \cite{Ho}, every tight contact structure on $L(3r-1,r)$ can be obtained by a Legendrian surgery on either  $(L(3,1), \zeta_1)$ or $(L(3,1), \zeta_2)$. If a tight $(L(3r-1,r), \xi$) is obtained from a tight $(L(3,1), \zeta_1)$, 
Corollary \ref{back} rules out a Legendrian link surgery from $(L(3r-1,r), \xi$) to $(L(3,1), \zeta_1)$. Theorem \ref{HFlo} rules out a Legendrian 
knot surgery from $(L(3r-1,r), \xi$) to both $(L(3,1), \xi_1)$ and $(L(3,1), \xi_2)$.

\end{example}


\section{Obstructions from the Thurston--Bennequin framing and known knot surgery results} \label{tb}

The Legendrian surgery coefficient is one less than the Thurston--Bennequin framing of the Legendrian knot. Thus one can 
try to obtain obstructions to existence of Legendrian knot surgeries between two manifolds by enumerating knots with given integral 
surgeries, and then using bounds on the Thurston--Bennequin framing of these knots. While  both of these steps are difficult in general, some 
special cases follow as immediate consequences of known results. We collect them in this section; it is interesting 
to compare these with the results from preceding sections. (We prove no new integral surgery results
or Thurston--Bennequin bounds here.)

Here and in section \ref{otwist}, we will use (+1) contact surgeries along with Legendrian surgeries. (See \cite{DGS} for a detailed treatment
of the material reviewed in this paragraph.)  Recall that a $+1$ surgery is an operation that
can be thought of as inverse to Legendrian surgery; more precisely, if $K$ is a Legendrian knot and $K'$ is its Legendrian push-off, then a Legendrian surgery on $K$ and a $+1$ contact surgery on $K'$ cancel each other. Every contact 3-manifold can be obtained from $(S^3, \xi_{std})$
by Legendrian and +1 contact surgeries on components of some link.  Unlike Legendrian surgery, $+1$ surgery does not 
always preserve Stein fillability or other similar properties of contact structures, and may produce both tight and overtwisted results. 
For example, +1 surgery on an unknot with $\tb=-2$ (Figure \ref{shark}) in the standard tight $S^3$ yields an overtwisted $S^3$. We will 
often encounter this particular overtwisted $S^3$, and will refer to it as the surgery on the ``shark'' (alluding to the shape of the 
stabilized unknot).  A contact surgery diagram for a contact structure
 allows to compute its 3-dimensional homotopy invariant $d_3$  via the formula 
\begin{equation}\label{d3}
d_3 (\xi) = \frac{c_1(\s)^2-2\chi(X)-3 \sign(X)+2}4 +m,  
\end{equation}
where $X$ is a 4-manifold bounded by $Y$ and obtained by adding $2$-handles to $B^4$ as dictated by the surgery 
diagram, $\s$ is the corresponding $\Spinc$ structure on $X$, and $m$ is the number of $(+1)$-surgeries in the diagram.
The $\Spinc$ structure $\s$ arises from an almost-complex structure defined in the complement of a finite set in $X$, and 
\begin{equation} \label{c1rot}
\langle c_1(\s), h \rangle = \rot (K)
\end{equation}
where $h$ is a homology generator of $X$ corresponding to the handle attachment 
along an (oriented) Legendrian knot $K$. 
  
(Here we assume that $c_1(\xi)$ is torsion; recall  that in this case the invariant $d_3$ and the  $\Spinc$ structure on $Y$ induced by $\xi$ 
determine the homotopy class of the plane field $\xi$.) For instance, for the surgery on the shark from Figure \ref{shark} we get $d_3=\frac{1}{2}$.

\begin{figure}[htb] 
\includegraphics[scale=0.8]{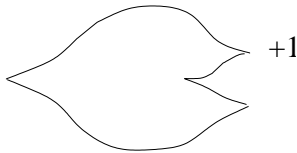}
\caption{The +1 surgery on the shark gives an overtwisted $S^3$ with $d_3=\frac{1}{2}$.}\label{shark} 
\end{figure}  
 
 We first consider surgeries on the unknot in  a tight or overtwisted $S^3$. 
In the tight $(S^3, \xi_{std})$, Legendrian surgeries are restricted by the fact 
that $\tb(\unknot)\leq -1$. If we want to get a tight result by a surgery in an overtwisted $S^3$, we need to consider knots with tight 
complements (i.e. {\em non-loose} knots). Non-loose unknots in overtwisted  $S^3$ were classified in \cite{EF}; they only exist in 
$(S^3, \xi_o)$ with $d_3(\xi_o)=\frac12$, can only have values of $(\tb, \rot)$ equal to $(n, \pm(n-1))$ with $n>0$, and are classified 
by $\tb$ and $\rot$. An explicit construction of the non-loose knots in $(S^3, \xi_o)$  was first given in \cite{Dym}. We find it useful to give a contact surgery description of these unknots.

\begin{figure}[htb] 
\includegraphics[scale=0.8]{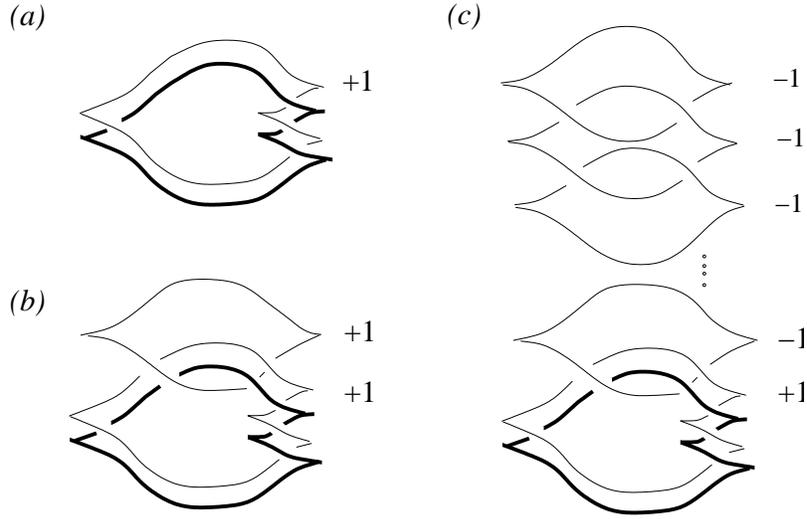}
\caption{Non-loose unknots with (a) $\tb=+2$, (b) $\tb=+1$, (c) $\tb\geq 3$.}\label{unknots} 
\end{figure}  

\begin{prop} The contact surgery diagrams shown on Figure \ref{unknots} represent $(S^3, \xi_o)$. The Legendrian knots (shown by thicker lines)
represent non-loose unknots with $\tb=n$ and $\rot=\pm(n-1)$ for all $n>0$.   
\end{prop}
\begin{figure}[htb] 
\includegraphics[scale=0.915]{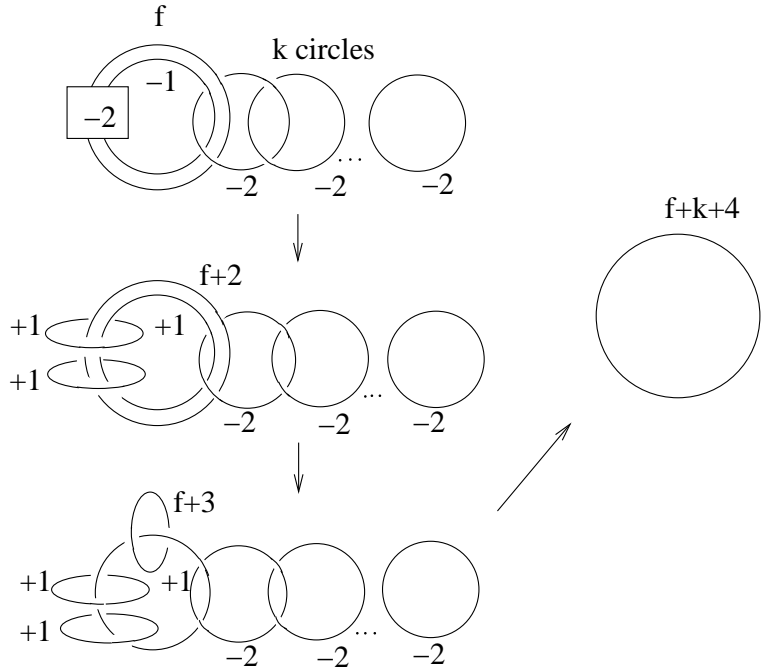}
\caption{Computing the framings. Unlike Figure \ref{unknots}, the numbers indicate Seifert framings, not those relative to the contact structure.}\label{kirby} 
\end{figure}  

\begin{proof} First, observe that the manifold shown is $S^3$, the contact structure is overtwisted and has $d_3=+\frac12$, 
thus we have $(S^3, \xi_o)$. Legendrian surgery on each thickly drawn knot cancels the +1 contact surgery on the shark and produces 
a tight contact manifold; it follows that all of these knots are non-loose. We compute the Thurston--Bennequin number:
 if we ignore the contact surgeries for  $(S^3, \xi_o)$ (i.e. consider the thick knots as living in $S^3$),
the contact framing on each of these unknots is 2 less than the Seifert framing from the disks they bound in  $(S^3, \xi_{std})$ (in other words,
$\tb=-2$). A few Kirby moves (see Figure \ref{kirby} for the case of non-loose knots with $\tb\geq 3$) show that  
the framing increases in the surgered manifold, so that the unknot in the diagram with $k$ Legendrian surgeries on Figure \ref{unknots}(c)
has $\tb=-2+k+4=k+2$. Cases (a) and (b)  on Figure \ref{unknots} are similar.  

\end{proof}

\begin{prop}  If $p>2$, a tight $L(p,1)$ cannot be obtained from an overtwisted $S^3$ by 
a surgery on a Legendrian unknot. ${\RR} \text{P}^3$ with its unique tight contact structure 
$\xi_t$ can be obtained from both tight and overtwisted $S^3$.
 \end{prop}

\begin{proof} The first part follows immediately from examination of the Thurston--Bennequin numbers of non-loose knots. 
As for the second part, a surgery on the unknot with $\tb=-1$ and $\rot=0$ produces $(\RR \text{P}^3, \xi_{t})$ from $(S^3, \xi_{std})$. 
To get $(\RR \text{P}^3, \xi_{t})$
from an overtwisted sphere, do surgery on one of the two non-loose unknots with $\tb= +3$ in $(S^3, \xi_o)$. Note that these  surgeries exhaust all 
possibilities: by \cite{KMOS} the only surgery in $S^3$ that produces  $\RR \text{P}^3$ is  $\pm 2$ surgery on the unknot. The unknot then has to be non-loose 
and have $\tb=-1$ or $+3$, which means that the contact structure on $S^3$ must be $\xi_{std}$ or $\xi_{o}$.    
\end{proof}

We also have
\begin{prop} \label{-L4n+3} A tight $-L(4n+3, 4)$ and a tight $S^3$ cannot be obtained from each other by a single Legendrian surgery (in any direction).
In addition, a tight $S^3$ cannot be obtained from an overtwisted  $-L(4n+3, 4)$.  
\end{prop}

\begin{proof} By \cite{Ras}, the only integral surgery that produces $-L(4n+3, 4)$ from $S^3$ is the $4n+3$ surgery on the positive $(2n+1, 2)$
torus knot $T_{2n+1, 2}$. The maximal Thurston--Bennequin number of  $T_{2n+1, 2}$ in the tight $S^3$ equals to $2n-1$, thus no $\pm 1$ 
contact surgery on this knot can produce $-L(4n+3, 4)$.  
\end{proof}

\section{On overtwisted contact structures} \label{otwist}

A result from \cite{EH} asserts that a Stein cobordism from $(Y_1, \xi_1)$ to $(Y_2, \xi_2)$ exists whenever $(Y_1, \xi_1)$ is overtwisted.
We now examine the relation between the second Betti numbers of such Stein cobordisms
and smooth cobordisms between $Y_1$ and $Y_2$. When working with overtwisted contact structures, we will often use Eliashberg's classification result that says that two overtwisted contact structures are isotopic whenever they are homotopic 
as plane fields. 

\begin{prop} \label{ot-sur} If $Y_2$ can be obtained from $Y_1$ 
by an integral surgery on an $n$-component link, then an overtwisted contact 
structure on $Y_2$ can be obtained
from an overtwisted contact structure on $Y_1$ by a Legendrian surgery on a link with the same number of components.
Moreover, one of the contact structures ($\xi_1$ on $Y_1$ or $\xi_2$ on $Y_2$) can be fixed. 
\end{prop} 

\begin{proof} We first consider the case where $Y_2$ is obtained from $Y_1$ by a surgery on a single knot. 
Suppose that $K$ is a knot in $Y_1$ that has an integral surgery producing $Y_2$,  and  let $l$ be the choice of longitude of $K$ 
that corresponds to the surgery framing. More precisely, we think of $l \in H_1(\d (Y_1 - \nu K)) $ as the homology class of a curve 
on the torus such that $m \cdot l = 1$, where $m \in H_1(\d (Y_1 - \nu K))$ is the class of the meridian of the knot. 
(Here the torus $\d (Y_1 - \nu K)$
is oriented as the boundary of the tubular neighborhood $\nu K$). The surgery producing $Y_2$ is then given by attaching a solid torus to 
$Y_1 - \nu K$ as dictated by the framing $l$. Now in the contact manifold  
$(Y_1, \xi_1)$, isotope $K$  to make it Legendrian, 
and let $tb \in H_1(\d (Y_1 - \nu K))$ stand for the framing on $K$ induced by $\xi_1$.  
To perform Legendrian surgery, we need to have $l=tb-m$. If $l=tb-k\,m$ with $k>1$,
we can decrease the Thurston--Bennequin framing by stabilizing the knot. (To make sure that the result of the surgery will be overtwisted, 
we can connect-sum, away from the knot,  with an overtwisted $S^3$ with $d_3=-\frac12$. This won't change the homotopy, and thus the isotopy type of $\xi_1$.) 
If $l=tb+k\,m$ with $k\geq 0$, we will increase the Thurston--Bennequin number as shown on Figure \ref{overtwist}. As before we first take the connected sum of $\xi_1$ and an overtwisted  $S^3$ with $d_3=-\frac12$, where the latter  is  shown on Figure \ref{overtwist}.  
This $S^3$ is in turn the connected sum of  $(S^3, \xi_o)$ which is 
the $(+1)$ contact surgery on the shark, and an $S^3$ with $d_3=-\frac{3}2$ which is 
the contact surgery on the two-component link  shown in the figure. (See \cite{DGS} for the detailed calculation of the  homotopy invariants). 
Taking a connected sum of $K$ with an unknot with $\tb=+2$
that lives in $(S^3, \xi_{o})$ gives a knot with the Thurston--Bennequin framing   $tb(K)+3m$. We  can repeat this procedure (together with extra stabilizations if needed)
to get the required value of $tb$. (The fact that the Thurston--Bennequin number of a knot can be increased in overtwisted contact 
structures is well-known; in fact this can be achieved by taking connected sums with the boundary of the overtwisted disk, see \cite{E1}.  We 
 included a  contact surgery version of such construction for completeness of discussion.) 
  
Observe that in this process we did not change the isotopy type of $\xi_1$ (but have little control over the isotopy type of $\xi_2$). 
To obtain a given overtwisted contact structure $\xi_2$ on $Y_2$, we can  apply a similar procedure 
for $+1$ contact surgery on the knot dual to $K$ in $(Y_2, \xi_2)$. 

To prove the statement for links with multiple components, iterate the above argument.
\end{proof}
\begin{figure}[htb] 
\includegraphics[scale=0.8]{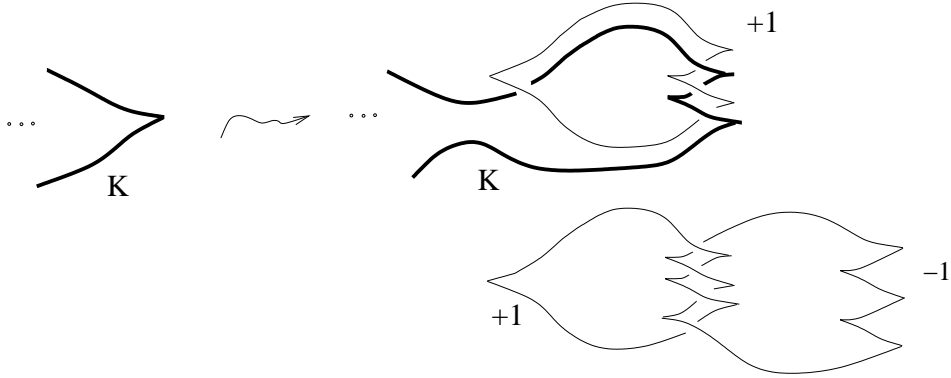}
\caption{Matching the contact and the surgery framings: increasing $tb$.}\label{overtwist} 
\end{figure}  

\begin{remark} Note that we cannot hope to fix both $\xi_1$ and $\xi_2$: suppose that $Y_2$ is obtained from $Y_2$ by a surgery on a knot, and observe that two arbitary homotopy classes of contact structures 
cannot always be connected by a single handle attachment. (Examples are easy to find on $S^3$.)  
\end{remark}

\begin{prop} \label{cando} Every overtwisted contact manifold  can be obtained from itself by a Legendrian surgery 
on an appropriate Legendrian unknot. 
\end{prop}

\begin{proof} 
First consider $S^3$ with the overtwisted contact structure $\xi_o$ that has $d_3(\xi_o)=+\frac{1}{2}$. 
There is a  Legendrian unknot with $\tb=0$ and  $\rot= 1$: for example, we can stabilize the non-loose
unknot with $tb=+2$ and $\rot=-1$ to decrease $\tb$ and increase $\rot$.    
From  (\ref{d3}),  the Legendrian surgery on this unknot $U$ results 
in a contact structure homotopic to $\xi_o$ (and necessarily overtwisted, since $d_3=+\frac12)$; but then the resulting 
contact structure is isotopic to $\xi_o$. To find an unknot with  $\tb=0$ and  $\rot= -1$ in any other overtwisted
contact structure $\xi$, we represent $\xi$ as the connected sum of itself and the overtwisted contact structure on $S^3$ 
with $d_3=-\frac1{2}$ (as in Proposition \ref{ot-sur} and Figure \ref{overtwist}), and then find a copy of the required unknot in the $(S^3, \xi_o)$
part of the contact structure. 

\end{proof}

We now turn to the case where the convex end of the cobordism may be tight.

\begin{lemma} Suppose $(Y_2, \xi_2)$ can be obtained from $(Y_1, \xi_1)$ by a Legendrian surgery. Then there is an overtwisted contact 
structure $\xi'_2$ on $Y_2$ and a Legendrian surgery that produces $(Y_1, \xi_1)$ from $(Y_2, \xi'_2)$.
\end{lemma}

\begin{proof} Suppose that $K$ is a Legendrian knot in $(Y_1, \xi_1)$ such that $(-1)$ contact surgery on $K$ produces 
$(Y_2, \xi_2)$. If we stabilize $K$ twice, the same integral surgery  becomes a  +1 contact surgery and results in $Y_2$ with a different 
contact structure $\xi'_2$. By \cite{LS}, a +1 surgery on a stabilized Legendrian knot always yields an overtwisted contact structure. 
Thus $\xi'_2$ will be overtwisted; cancelling the +1 surgery by a Legendrian surgery, we recover $(Y_1, \xi_1)$.    
\end{proof}

\begin{cor} If some contact structure on $Y_2$ can be obtained from a tight contact $Y_1$, then the same tight contact structure on $Y_1$
can be obtained by a single Legendrian surgery from an overtwisted contact structure on $Y_2$.
\end{cor}

\begin{prop} Suppose that $Y_2$ can be obtained from $Y_1$ by an integral surgery on an $n$-component link.  Then any tight contact structure $\xi_2$ on $Y_2$
can be obtained from some overtwisted $(Y_1, \xi_1)$  by a Legendrian surgery on a link with $n+1$ components. 
\end{prop}

\begin{proof} First do a  +1 contact surgery on a shark in $(Y_2, \xi_2)$ to obtain an overtwisted $(Y_2, \xi'_2)$. By  Proposition \ref{ot-sur},
a Legendrian surgery on some $n$-component link in an overtwisted $(Y_1, \xi_1)$ produces  $(Y_2, \xi'_2)$.
 Now, do  a Legendrian surgery on a push-off 
of the shark to undo the +1 surgery and recover  $(Y_2, \xi_2)$.
\end{proof}

Proposition \ref{-L4n+3} gives an example where the result of the previous proposition cannot be improved.

We conclude with a slightly more interesting 
example: Legendrian surgeries between various contact structures on $S^3$ and $L(7,4)$. Recall  by \cite{Ho} there are only 
three tight contact structures on $L(7, 4)$, and they are given by Legendrian surgeries depicted in Figure \ref{L74}. 
The values of the $d_3$-invariant can be computed to be 0 for the contact structure $\Xi_0$ given by the more symmetric surgery diagram, and $-\displaystyle{\frac{2}{7}}$ for the other 
two contact structures $\Xi_1, \Xi_2$. 
\begin{figure}[htb] 
\includegraphics[scale=0.7]{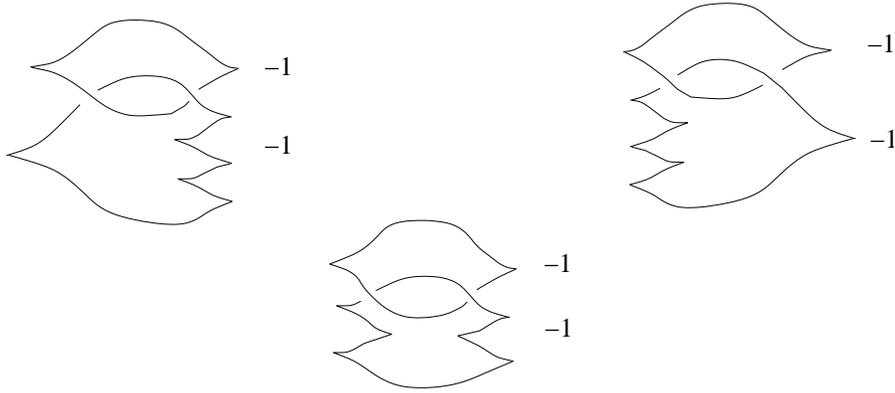}
\caption{Tight contact structures $\Xi_0$ (middle), $\Xi_1$ (left), and $\Xi_2$ (right) on $L(7, 4)$.}\label{L74} 
\end{figure}  

\begin{prop} Although  $(L(7, 4), \Xi_0)$ can be obtained from $(S^3, \xi_{std})$, the contact manifolds  $(L(7, 4), \Xi_1)$ and 
$(L(7, 4), \Xi_2)$ cannot be obtained by a surgery on Legendrian knot from any contact structure on $S^3$, tight or overtwisted. 
\end{prop}

\begin{proof} First, by \cite{Ras} we know that the integral surgery that can produce $L(7, 4)$ from $S^3$ has to be $-7$ surgery on the 
left-handed trefoil. If we start with a tight $S^3$, there is a unique Legendrian left-handed trefoil with $\tb=-6$ (by the classification 
of Legendrian torus knots \cite{EH1}). This trefoil has $rot=0$, so the Legendrian surgery produces  $(L(7, 4), \Xi_0)$. 

If we want to produce the contact structure $\Xi_1$ or $\Xi_2$ from an overtwisted $S^3$, we need to look for a non-loose 
left-handed trefoil with $\tb=-6$. Notice that for such a knot, the Thurston--Bennequin bound will be satisfied: indeed, $tb$ and $rot$
of a null-homologous non-loose knot are restricted by 
the inequality
$$
-|\tb(K)| +|\rot(K)|\leq 2g-1.
$$
(see \cite{Dym}, Proposition 5.3 in the arxiv version).

It follows that $|\rot(\text{trefoil})| \leq 7$. To pin down the value of $\rot$, notice that 
siince contact structures on $S^3$ have half-integer values of $d_3$, by (\ref{d3}) and (\ref{c1rot})  we have 
$$
\frac{m}2+ \frac{c_1^2-2+3}{4} = -\frac{2}{7}, \qquad \qquad c_1^2= - \frac{\rot^2}{7}
$$  
for some integer $m$. Thus, 
$\displaystyle{\frac{\rot^2-8}{7}=2m+1}$ is an integer, $\rot^2 \equiv 1 \mod 7$, and then $\rot=\pm 1$, $m=-1$.
This means that a Legendrian surgery  in $S^3$ can produce   $(L(7, 4), \Xi_1)$ or $(L(7,4), \Xi_2)$
only if  $S^3$ is equipped with the  contact structure $\xi_{otw}$ with $d_3(\xi_{otw})=-\displaystyle{\frac{1}2}$. 

Our next claim is that the contact structure $\xi_{otw}$ on $S^3$ in this case would have to be tight rather than overtwisted.
(This leads to a contradiction, proving the proposition.)
 Indeed, for $\Xi=\Xi_1$ or $\Xi_2$, we show that Heegaard Floer contact invariant 
is non-zero for the contact structure obtained from  $(L(7, 4), \Xi)$ by a $+1$ contact surgery on the knot dual to the left-handed trefoil.
Consider the map 
$$
F_{X, \t}: \HF(-L(7, 4), \s) \to \HF(-S^3)
$$
that is induced by the surgery cobordism $X$ equipped with the $\Spinc$ structure $\t$ induced by the Stein structure on the 2-handle.  The oriented smooth 4-manifold $X$ with boundary that gives this cobordism from  $-L(7, 4)$ to $-S^3$
can also be viewed as the cobordism from $S^3$ to $L(7,4)$ given by our Legendrian surgery, i.e   the topological $-7$ surgery on the left-handed trefoil.
  Since $H_2(X)$ is generated by a surface   
with self-intersection $-7$, $X$ is negative definite. Since $S^3$ and $L(7, 4)$ are both $L$-spaces, 
by \cite[Lemma 2.5]{LS1}
the map  $\widehat{F}_{X, \t}$ will be injective as long as its degree shift equals to $d(-S^3) - d(-L(7, 4), \s)$, 
where the $\Spinc$ structure $\s$ is the restriction of $\t$ to $L(7, 4)$.
But  $d(-S^3) - d(-L(7, 4), \s) = d(L(7, 4), \s) - d(S^3) =  d_3(\Xi) - d_3(\xi_{otw})$, since we proved that $d_3(\xi_{otw})= -\frac{1}{2}= d(S^3)-\frac{1}{2}$, and $d_3(\Xi)= d(L(7, 4), \s)- \frac{1}{2}$ because $\s$ is induced by $\Xi$. We know that 
$$
d_3(\Xi) - d_3(\xi_{otw}) = \frac{(c_1(X, \t))^2 -2 \chi(X) - 3 \sigma(X)}{4}
$$
since $(X, \t)$ is a Stein cobordism between the two contact manifolds; on the other hand, the latter quantity is the degree shift of $F_{X, \t}$.
Therefore,  $c(\xi_{otw})=F_{X, \t}(c(\Xi))$  is non-zero. 
\end{proof}

  \end{document}